\documentclass[12pt]{article}
\usepackage{hyperref}

\newcommand{\note}[1]{}

\renewcommand\marginpar[1]{}

\usepackage{amsmath} 
\newcommand{\comment}[1]{}
\newcommand{\complexes}{{\bf C}}

\providecommand{\qed}{\vrule height 6pt depth 0pt width 3 pt}

\newcommand{\reals}{{\bf R}}

\newcommand{\rank}{\mathop{\rm rank}\nolimits}

\newcommand{\mspan}{\mathop{\rm span}\nolimits}

\newcommand{\BibTeX}{{\rm B\kern-.05em{\sc i\kern-.025em b}\kern-.08em     
    T\kern-.1667em\lower.7ex\hbox{E}\kern-.125emX}}

\newcommand{\ball}[2]{B_{#2}(#1)}

\newenvironment{proof}[1][Proof]{\begin{trivlist}\item[\hskip \labelsep
{\it #1. }]}{\hfill \qed \goodbreak \end{trivlist}}

\numberwithin{equation}{section} 
\newtheorem{theorem}[equation]{Theorem}
\newtheorem{proposition}[equation]{Proposition}
\newtheorem{corollary}[equation]{Corollary}
\newtheorem{lemma}[equation]{Lemma}

\begin{document}

\title{Estimates for a family of multi-linear forms}

\author{
Zhongyi Nie
\\ Department of Mathematics 
\\University of Kentucky
\\ Lexington, Kentucky 
\and 
   Russell M. Brown
\\ Department of Mathematics 
\\University of Kentucky
\\ Lexington, Kentucky }

\date{}

\maketitle

\abstract{ 
We consider a special class of the multi-linear forms studied by
Brascamp and Lieb.  For these forms, we are able to characterize the
 $L^p$ spaces  for which the form is bounded. We use this
characterization to study a non-linear map that arises in scattering
theory. 

{\em Keywords: }Multi-linear interpolation, scattering theory

{\em Mathematics subject classification: }26D20, {\em Secondary: } 37K10,  46B70 
}

\
\note
{ Index of notation.

\begin{tabular}{rl}

\sc Symbol &  Meaning \rm \\
$k$ & the form lives on $k$ copies of $ \reals ^ \ell$ \\
$m  $ & the number of functions in our form. \\
$\Lambda$ & the general form \\
\end{tabular}
}

\section{Introduction}

In this note, we consider a family of multi-linear forms involving
fractional integration and establish estimates for these forms on
products of $L^p$-spaces. Using these estimates, we are able to give a
proof of continuity of a scattering map in two dimensions. This
scattering map
may be found in work of Fokas \cite{AF:1983}, as well as later work of
several authors including Fokas and Ablowitz \cite{FA:1984}, Beals and
Coifman \cite{BC:1988}, and Sung
\cite{LS:1994a,LS:1994b,LS:1994c}. These authors were interested in a
two-dimensional scattering theory that served to transform solutions
of one of the Davey-Stewartson equations, a nonlinear evolution
equation in two space dimensions, into solutions of a linear
system. The map reappeared in work of Brown and Uhlmann \cite{BU:1996}
on the inverse conductivity problem.  In the inverse conductivity problem, we
are interested in recovering a conductivity coefficient from the
Dirichlet to Neumann map. As part of this recovery, it is 
interesting to know something about the continuity properties of the
scattering map. This was one motivation for the work of Brown
\cite{RB:2001b}.  This work of Brown shows that the scattering map is
continuous in a neighborhood 0 in $L^2$. In this article, we provide a
new proof of some of the results of Brown and  give a description  of
the set of $L^p$ spaces where certain multi-linear forms are bounded.
This description appeared earlier in work of Barthe
\cite[p.~348]{FB:1998}.   

To describe our main  result in more detail, for $n=0,1,2,\dots$ we consider
the  multi-linear form 
\begin{eqnarray}
\lefteqn{\Lambda_n (t, q_0, q_1,\dots, q_{2n} ) }\nonumber \\
&=&  \int _ { \complexes ^ { 2n +1} } \frac { t(\sum _ {k = 0 } ^ { 2n} (
  -1) ^ k x _k ) \prod_{ k =0} ^ {2n} q_ k ( x_k ) } 
{ ( x_ 0-x_1) ( \bar x_ 1 - \bar x_2) \dots ( \bar x_{ 2n-1} - \bar x_
  {2n})}\, dx_0 \, dx_1\dots dx_{ 2n}. \label{MainFormDef}
\end{eqnarray} 
In this expression, we are using $x_j$ to stand for a complex variable
and $dx_j$ denotes Lebesgue measure on the complex plane.
Our goal is to show that  there is constant $c$ so that 
\begin{equation}\label{MainFormEstimate}
\Lambda_n (t, q_0, q_1, \dots, q_{2n})\leq c^n \|t\|_{1/2} \prod _ { j =0
} ^ {2n} \| q_j \|_{1/2}.
\end{equation}
Here and throughout this paper, we will denote the $L^p$ norm of a
function $f$ by $\|f\|_ {1/p}$ with the  convention that $
1/\infty$ is 0. 
Thus, we  provide a new proof of the main estimate in the article
\cite{RB:2001b}, but without the precise dependence of the constant. 
As in Brown's work \cite{RB:2001b}, this leads to the continuity of the scattering
map on $L^2$.  
The method is perhaps a bit more flexible and we are able to give an
extension of these results when some of the functions come from  $L^p$
spaces  for $p\neq 2$.  We use this to obtain 
an analogue  of the Hausdorff-Young inequality for the scattering
map. 

We briefly describe the results of this paper. Most  of the results of
this paper
first appeared in the Ph.D.~dissertation of the author Nie \cite{ZN:2009}.
The first part of our paper considers general multi-linear forms 
\begin{equation}
\label{GeneralForm}
\Lambda(a_1, a_2, \dots, a_m) = \int _ { \reals ^ { k\ell }} \prod_{ j =
  1} ^ m  a_j(f_j\cdot
x)\, dx, \qquad a_j \in L^1( \reals ^\ell ) \cap L^ \infty(\reals ^\ell)
\end{equation}
where $ x\in \reals ^ { k\ell} $, $ x= (x_1, \dots x_k) $ and each $ x_i
\in \reals ^ \ell $, $ f_ j \in  \reals ^ k$ and $M= \{ f_1, f_2 ,
\dots f_m\}$ is a collection of vectors in $ \reals ^ k$. We define
$ f_j \cdot x = \sum _ { i =1 } ^ k f_{ji} x_i$. We will 
 arrive at the form $ \Lambda_n$ by setting some of the
functions $a_j$ to 
 be $1/x$ which lies in the Lorentz space $L^ { 2, \infty}(
\complexes)$.    Thus, we will be interested in estimates in Lorentz
spaces. 
We consider the set $ \Omega _ \Lambda$ which is defined to be the set
of $ \theta= ( \theta_1, \dots, \theta_m)$ for which we have the
inequality 
\begin{equation}\label{GeneralEstimate}  
\Lambda ( a_1 , \dots , a_m ) \leq C_ \theta \prod _{ j =1 } ^ m \|a_j
\|_ {\theta_j } 
\end{equation} 
for some constant $ C_ \theta$. We show that this set $ \Omega
_\Lambda$ is the matroid polytope (or, more precisely, the basis
matroid polytope) for the matroid formed by the set of vectors $M
\subset \reals ^k$.  Recall that a set of vectors $M$ and the
collection of linearly independent subsets of $M$ form a matroid. We
recall that 
the {\em matroid polytope }for $M$ (or basis matroid polytope for
$M$), $\Omega_M$, is the 
convex hull of the vectors $\{ \chi _B : B \mbox { is a basis for
  $\reals^k$}\}$.  
 We are using $ \chi_S$ to denote the indicator
function of a set $S \subset M$. Thus the $i$th component of $ \chi_S$
is 1 if $ f_i \in S$ and 0 otherwise.  The matroid polytope can also
be described by a set of inequalities and we are able to use this
description to establish our estimates. We refer to the monograph of
Oxley \cite{MR1207587} or the textbook of Lee \cite{MR2036957} for
basic facts about matroids. We will use two operations on sets in
matroids. For matroids given as a subset of a vector space, we may
define the {\em span }of a set $S \subset M$ as $ M \cap V$ where $V$
is the span of $S$ in the vector space. The {\em rank }of $S$ in the
matroid sense can be be defined as the dimension of the vector space
spanned by $S$.

The characterization of the set $ \Omega _ \Lambda$
as a matroid polytope may be found in the  work of Barthe
\cite[p.~348]{FB:1998}, though Barthe does not use the word matroid
polytope.

In passing from estimates in Lebesgue spaces to estimates in Lorentz
spaces, we make use of a multi-linear interpolation theorem of
S.~Janson \cite{SJ:1986}. 
  The constants in this theorem
depend on the distance from the boundary of $ \Omega_M$ and obtaining
the correct dependence on $n$ as $n$ tends to infinity requires
additional work.  
We do not attempt to summarize all of the
work related to multi-linear forms,  but refer readers to the  survey paper of Beckner
\cite{MR1315541}, 
 recent work by Bennett, Carbery,
Christ,  and Tao \cite{BCCT:2005,MR2377493}, Carlen, Lieb and Loss
\cite{MR2077162} as well as the earlier work of
Brascamp and Lieb \cite{MR0412366} for related work on multi-linear
estimates. Note that our work is much simpler in that we do not make
an effort to find the optimal estimate in our inequalities.  
Using estimates in Lorentz
spaces to 
obtain estimates for fractional integration dates back at least to 
 \cite{MR0146673} and Beckner \cite{MR1315541}  discusses forms
involving fractional  integration. 

Both authors  thank Jakayla Robbins for pointing out to us that the set
$ \Omega_M$ is a matroid polytope.

\section{Estimates in Lebesgue spaces}
\label{Lebesgue}

In this section, we continue to consider the form
(\ref{GeneralForm}). 
We begin with the following simple proposition.

\begin{proposition} \label{BaseBound}
If $ B \subset M$  is a basis for $ \reals ^ k$,
 and $ \chi _B = ( \theta _1 , \dots \theta _m)$,   then we have
$$
\Lambda ( a_1, \dots , a_m) \leq |\det \hat B | ^ { - \ell } 
\prod _ { j =1 } ^ m \| a_j \| _ { \theta_j } . 
$$
Here we are using $ \hat B$ to denote the $k\times k$ matrix whose
rows are the elements of  $B$. 
\end{proposition}

\begin{proof} We will make a change of variables in the integral
  defining $ \Lambda$. We let $ B = \{ f_{ i _1} , \dots , f_ { i _
    k}\}  $ and define $ y_j = f_{i_j} \cdot x$. If we make this change of
variables in the form $\Lambda$, the estimate of the Lemma becomes
obvious. To obtain the constant, we observe that the determinant of
the map $ x \rightarrow y$ on $\reals ^ {k\ell}$ is $|\det \hat B | ^ \ell$. 
\end{proof}

As noted above a  set of vectors $M = \{f_1, \dots, f_m\}$ gives a
matroid. 
Since each basis for $\reals ^k$ contains $k$ elements, the matroid
polytope for $M$, $\Omega _M$ lies in the 
hyperplane given by $ \sum_{ i=1 } ^ m  \theta _i =k$. 
As a corollary of this definition and the previous theorem, we have
the following. 
\begin{corollary} 
\label{LebesgueBound}
 For $ \theta \in \Omega _M$ and $\Lambda$ as defined
  in (\ref{GeneralForm}), we have
$$
\Lambda (a_1, \dots, a_m) \leq C \prod_ { i =1 } ^ m \| a_i\| _ {\theta _i }. 
$$
where $ C = \max \{ |\det \hat B|  ^ { -\ell}: B \subset M \mbox{ is a
    basis for $ \reals ^ k$}\}$. 
\end{corollary}

\begin{proof} The Corollary follows from Proposition \ref{BaseBound}
  and  the theorem on   complex multi-linear interpolation 
from  the monograph of Bergh and L\"ofstrom \cite[Theorem 4.4.1]{BL:1976}. 
\end{proof}

We observe that the constant in this estimate could  be 
improved. In our application, the determinant  will  be 1  at every
vertex and thus we choose to not dwell on the constant.

The converse of Corollary \ref{LebesgueBound} also holds. If 
estimate (\ref{GeneralEstimate}) holds for a finite constant, then 
 the point $ \theta$ lies in the matroid polytope, $
\Omega _M$.  This converse is not needed in our argument, but is
included for completeness. 
To establish the converse, we recall Theorem 2.1 in the work of Bennett
{\em et. al. }\cite{BCCT:2005}, specialized to the form in
(\ref{GeneralForm}).  
\begin{theorem} \cite[Theorem 2.1]{BCCT:2005} 
\label{Bennett}
We have the estimate 
(\ref{GeneralEstimate}) for $ \theta \in [0,1]^m$ if and
  only if we have 
$$ \sum  _{ i=1} ^m \theta _i  = k$$
and  for every subspace $V \subset \reals ^ { k\ell}$, we have
$$
\dim(V) \leq \sum _ { i =1} ^ m \theta _i \dim( f_i \cdot V).
$$
\end{theorem} 

\begin{corollary} If the form in (\ref{GeneralForm}) satisfies the estimate
  (\ref{GeneralEstimate})  for $ \theta   \in [0,1 ] ^m$, then we have
  $\theta \in \Omega_M$.  
\end{corollary}

\begin{proof}  It is known that the matroid polytope  can be described
as the set of $ \theta \in [0,1]^m$ which lie in the hyperplane $ 
\{\theta : \sum _i \theta _ i =k\}$ and which satisfy the 
inequalities
\begin{equation}\label{Rank}
\sum _{ \{i: f_i \in S\}} \theta _i \leq \rank(  S) 
\end{equation}
for all subsets $S \subset M$. 
See the
textbook of J.  Lee \cite[p.~67]{MR2036957}, for example. 

Assume the form satisfies the estimate (\ref{GeneralEstimate})
for  $ \theta$. Let $ S \subset M$ and we will show the above
inequality. 
Towards this end, we let  $V$ be the
orthogonal complement of $S$,  $V = S^ \perp$. 
Let $ {\cal V} = \{ ( v_1x, v_2x, \dots, v_mx): v \in V, x\in \reals ^
\ell\}$ and thus ${\cal V } \subset \reals ^ { k \ell}$. 
From
Theorem \ref{Bennett} we have 
$$
\ell\dim(V)= \dim ({\cal V}) \leq \ell \sum _ { i =1 } ^m \theta _i \dim(f_ i \cdot V) 
= \ell \sum _ {\{  i : f_i \cdot V \neq \{ 0 \} \} } \theta _i
$$
Using that $ k= \sum _ { i=1 } ^ m  \theta _ i = \dim (V) + \dim ( V^
\perp)$ and we arrive at the inequality,
$$
\sum _ { \{i : f_ i \cdot V = \{ 0 \} \} }  \theta _i   \leq \dim ( V^
\perp).
$$
We observe that $ \dim ( V^ \perp) =  \rank ( S)$ and the Corollary
follows.
\end{proof}

We now consider an  extension of these estimtes  to the Lorentz
spaces. This relies on an interpolation theorem for multi-linear
operators of S.~Janson \cite{SJ:1986}. Janson's theorem is based on the real
method of  interpolation and thus gives us Lorentz spaces as
intermediate spaces.  

We develop the notation needed to state Janson's result. For $j
=1,\dots, m$, we let $ \bar A_j = ( A_{j0}, A_{ j1})$, $j =1 , \dots,
m$ and $ \bar B = (B_0,B_1)$ be Banach couples and then $A_{i\theta,q}
= [A_{j0}, A_{ j1}]_{\theta,q}$ will be the real interpolation
intermediate spaces.  We consider multi-linear operators 
$$
T: \prod _{ j =1 } ^ m A_{j0} \cap A_ {j1} \rightarrow B_0 + B_1. 
$$
We fix real numbers $\alpha_0,\alpha_1, \dots, \alpha_m$ with $ \alpha_i
\neq 0 $ for $ i=1,\dots,m$ and define 
\begin{eqnarray*}
\Omega = \{ ( \theta _1, \dots, \theta_m) \in [0,1]^m :  ( \alpha
_0 + \sum_{ i =1} ^ m  \alpha_i \theta _i ) \in [0,1]  \mbox{ and }
T: \prod _{ j =1 } ^ m A_{j\theta_j,q_j}  \rightarrow
B_{\theta,q}  \\
\mbox{ for some } q,q_1, \dots, q_m  \mbox{ in  }  (0,\infty]
\}.
\end{eqnarray*}
The main results  of Janson are  that the set $ \Omega$ is convex and 
that in the interior of $ \Omega$, $T$ is bounded  on real
interpolation spaces.

A simple application of Janson's results is the following theorem on
multi-linear forms.  This result  depends on the duality properties of
Lorentz spaces which  may be obtained, for example,  from the general result on
duality for real interpolation spaces in Bergh and L\"ofstrom 
\cite[Theorem 4.7.1]{BL:1976}. 
We will apply the next theorem  result not to the forms $ \Lambda$ but to forms
that are obtained by fixing some of the arguments of $ \Lambda$. Thus,
we state a result for more general multi-linear forms.

\begin{theorem} 
\label{Janson}
Let $ \Lambda $ be a multi-linear form which is
  defined at least on $ (L^ 1 ( \reals ^\ell) \cap L^ \infty (
  \reals^\ell))^m$ and suppose that 
$$
\Lambda (a_1, \dots, a_m)\leq A \prod_{i=1}^m \| a _ i \| _ { \eta _i}
$$
for all $ \eta$ in  $ B(\theta , \delta)\cap \{ \eta : \sum _ { i =1 }
^ m \eta_i  = K\}$ for some $K$  and $ B(\theta, \delta) \subset [0,1]^ m$.  Then for $(q_1,
\dots, q_m)$ satisfying $ \sum _ { i = 1 } ^ m \frac 1 { q_i} \geq 1$,
we conclude that 
$$
\Lambda ( a_1, \dots, a_m) \leq C \prod _ { i =1} ^ m \| a_i \|_{
  \theta _i , q_i}.
$$
The constant $C$ depends on $ \theta$, $(q_1, \dots, q_m)$, $m$, and $
\delta$. 
\end{theorem}
\begin{proof} Since we assume that $ \theta $ is an interior point of
  the cube $[0,1]^m$, we have, in particular, that $0 < \theta _1 $. We
  define an $(m-1)$-linear operator $T$ by 
$$
  \int_{\reals ^ \ell} T(a_2, \dots, a_m)a_1\, dx  = \Lambda(a_1, \dots, a_m).
$$
Our assumption on $ \Lambda$ implies that we have that
$$
T: \prod _ { i =2 } ^ m L^ { 1/\eta_i} ( \reals ^k) \rightarrow L^ {
  1/ ( 1- \eta _1)}( \reals ^ k ),\qquad  \eta \in \ball \theta \delta \cap
\{ \eta : \sum _ i \eta _i = K \}.
$$
  Our hypotheses allow us to apply
Theorem 2 from the article  of S.~Janson  \cite{SJ:1986} and  gives us
that for $ q, q_2, \dots, q_m$ in $ [1, \infty]$, we have 
$$
\| T(a_2, \dots, a_m) \| _ {  1-\theta _1,  q } 
\leq C \prod _ { i =2 } ^m \| a_ i \| _ { \theta _i , q_i }  
$$
provided $ \sum _ { i =2 } ^ m 1/q_i \geq 1/q$. 
Recalling our definition of the operator $T$ and the extension of
H\"older's inequality to the Lorentz spaces, we obtain the estimate of
the Theorem. 
\end{proof} 

\section{Estimates for the form $ \Lambda _n$}

The rest of this paper is devoted to the study of the form $ \Lambda
_n$ defined (\ref{MainFormDef}). We will realize this form as a special
  case of the form introduced in (\ref{GeneralForm})  where some of
the arguments $a_j$  are 
taken from Lorentz spaces.   

In this section  we consider the form
(\ref{GeneralForm}) where the functions $a_j$ live on the complex
plane, the number of functions is $4m+2$ and the vectors $ f_j$ lie in
$ \reals ^ { 2m+1}$ and are defined by 
\begin{eqnarray*}
f_{2j-1} & = & e_j , \qquad j =1, \dots, 2m+1 \\
f_ {2j} & = & e_j - e_ { j+1} , \qquad j = 1, \dots, 2m\\
f_{4m+2}& = & e_1-e_2+\dots+e_{2m+1}.
\end{eqnarray*}
The number of elements in our matroid is no longer $m$, but
$4m+2$. The vectors $f_j$ are elements of $ \reals ^ { 2m+1}$ and 
hence the parameter $k$ in  (\ref{GeneralForm}) is $ 2m+1$ and the parameter
 $\ell  =2$ as we have identified the complex plane  $ \complexes$ with $ \reals ^ 2$. 
We let $M= \{f_j : j =1, \dots,4m+2\}$ and then  $\Omega_M$
will be the matroid polytope for $M$ as introduced in section \ref{Lebesgue}. We
will show that the point $ (1/2, \dots, 1/2)$ lies in the interior of
the set $ \Omega _M$. This implies the desired estimate for the form,
but without the stated dependence of the constant. The argument below
is needed  to show that the constant in (\ref{MainFormEstimate}) is of
 the form $c^n$.

We introduce a set $P_\delta$ which  we will show  lies in
$\Omega_M$. For $ \delta >0$,  we let
\begin{eqnarray*}
P_\delta&  =  & \{ \theta \in \reals ^ { 4m+2} : \sum _ { i = 4j+1 } ^ { 4j+4}
\theta _i = 2, \mbox{ for } j=0,\dots,m-1,  \\
& & \qquad   \theta _ {4m+1} + \theta _{
  4m+2} = 1, \  | \theta _i -1/2| \leq \delta, \ i = 1,\dots,
4m+2\}. 
\end{eqnarray*}
\begin{theorem} \label{Inside}
If $ \delta \leq 1/10$, then $P_\delta \subset \Omega
  _M$.
\end{theorem} 

The proof begins with a few technical lemmata. In the following
discussion, we will let   $ B_k = \{ f_{ 4k+1} ,
\dots , f_{ 4k+4}\}$, for $k =0, \dots, m-1$ denote a block of 4 vectors. In addition, it will be useful to view the set
$M$ as an ordered set and for $j\leq k$, let $[f_j, f_k ] = \{ f_i : j
\leq i \leq k\}$ denote an interval in $M$. 
\begin{lemma} \label{SegEst}
 If $ S= [ f_{ 2k-1}, f _ { 2k+2j -1}] = [e_k , e_{
      k+j}]$ is an interval in $M$, then we have
$$
\rank (S) \geq   1/2 -3 \delta + \sum _ {f _i \in S } \theta _i  .
$$
\end{lemma}
\begin{proof} The proof proceeds by considering the four cases that
  arise when $k$ and $j$ are even and odd. 

{\em Case 1. } Let  $k$ be even and $j$ be even. 

In this case, $ S= \{ e_k, e_k -e_{ k+1} \} \cup ( \cup _ { i =k/2}
^ { (k+j)/2 -2} B_i )\cup \{ e_ { k+j-1}, e_ { k+j -1} - e_ { k+j } ,
e_  {k+j}\}$. 
It is clear that the (vector space)  span of $S$ is the subspace spanned by $ e_k, e_
{ k+1}, \dots, e_ {k+j}$ and thus $ \rank(S) = j+1$. We now consider
$ \sum _{ f_i \in S} \theta _i$. Each block contributes $ 2$ to the
sum. From the definition of $P_\delta$, we have $ \theta _ { 2k-1 } +
\theta _{ 2k}  \leq 1 + 2\delta$. Again, from the definition of
$P_\delta$, we have that 
$\theta _ { 2k+2j-3 }  + \theta _ { 2k+2j-2 }  + \theta _ { 2k+2j-1
}= 2  - \theta _ { 2( k+j)} \leq 3/2 +\delta$.  Thus, we have
$$
\sum _ { f_i \in S} \theta _i \leq  ( j+1 ) - 1/2 + 3 \delta.
$$
The conclusion of the lemma follows from this upper bound and the 
observation that $ \rank(S) = j+1$. 

{\em Case 2. } Let $k$ be even and $j$ be odd. 

We have $ S = \{e_k, e_k - e_ { k+1} \} \cup   ( \cup _ { i = k/2 }
^ { ( k+j -3)/2} B _ i )\cup \{ e_ { k+j} \} 
$
and again $ \rank(S) = j+1$. We have
$(j-1)/2$ blocks in $S$. If $ \theta \in P_\delta$,  we may use   the
upper bound of $1/2 + \delta$ 
for  the $\theta_i$ that do not correspond to blocks and obtain that
$$
\sum _ { f_i \in S} \theta _i \leq  ( j+1 ) - 1/2 + 3 \delta.
$$

{\em Case 3. } Let $k$ be odd and $j$ be even. 

In this case   we have $ S = (\cup _ { i = (k-1)/2} ^ { ( k+j-3)/2} B_i
) \cup \{ e_ { k+j}\}$.   As we have $j/2$ blocks and one extra vector,
it is easy to obtain the upper bound
$$
\sum _ { f_i \in S} \theta _i \leq  ( j+1 ) - 1/2 +  \delta.
$$
As $\rank(S) = j+1$, the estimate of the Lemma follows.

{\em Case 4.} Let $k$ be odd and $j$ be odd. 

In this case  we have 
$$ S = ( \cup _ { i = (k-1)/2 } ^ { ( k+j)/2 -2} B_i ) \cup \{ e _ { k
  + j -1} , e _ { k+j-1} - e_ { k+j} , e _ { k+j} \} .
$$
We have $(j-1)/2$ blocks and three extra vectors, thus we have
$$
\sum _ { f_i \in S} \theta _i \leq  ( j+1 ) - 1/2 + 3 \delta. 
$$
As $\rank(S) = j +1$, the estimate follows again. 
\end{proof}

\begin{lemma}  
\label{FindTriples}
Let $S \subset M \setminus \{ e_1-e_2 + e_3-e_4+\dots e_
  { 2m+1}\}$ be a dependent set. Suppose that $ \mspan (S) =
  S$, then we have that $S$ contains a set of  the form $ \{ e_k , e_k
  -e_{ k+1} , e_ { k+1}\}$ for some $k$. 
\end{lemma}

\begin{proof} 
Suppose that $S$  contains no  set of the form $ \{ e_k , e_k -e_{
  k+1}, e_ { k+1}\}$. Because $ \mspan(S) =S$, it follows that $S$
contains at most one element from each the sets $ \{ e_k, e_k - e_{
  k+1} , e_ { k+1}\}$, $k=1,\dots,2m$.  This  contradicts our  assumption  that $S$ is
a dependent set.  
\end{proof}

\begin{lemma}  
\label{Structure}
Let $ S \subset M \setminus \{ e_1-e_2+\dots + e_ {
    2n+1}\}$. If $ \mspan(S) =S$, then we may write 
$$
S =  \bigcup _ { i=0} ^ k S_i 
$$
where the collection $\{ S_i\}$ is pairwise  disjoint, for each
$i=1,\dots, k$, $S_i = [e_{s_i}, e_{ t_i}]$ is an interval  and  the set
$ S_0$ is independent.  For this decomposition, we have
$$
\sum _ { i = 0 } ^ k \rank(S_i) = \rank(S). 
$$
\end{lemma}

\begin{proof} If $S$ is linearly dependent, then by Lemma
  \ref{FindTriples}, we may find an index $k$ so that $ \{ e_k, e_k -
  e_{ k+1}, e_ {k+1}\}$ lies in $S$. Since $\mspan(S) = S$, if $ 
 \{ e_k, e_k -  e_{ k+1}, e_ {k+1}\}  
\subset S$   and $ e_ { k-1}$ lies in $S$, then  $e_{k-1}-e_k$ also
lies in $S$. Similarly, either  $ e_ {k+1} $ and $ e_ {k+1}-e_{ k+2}$
both lie  in $S$ 
or both do not lie in $S$. 
We let $ S_1$ be the maximal interval  of the form $[e_s,e_t]$ which
contains $\{ e_k, e_k-e_{ k+1}, e_{ k+1}\}$. 
It is clear that we have $\rank(S) = \rank(S_1) + \rank ( S 
\setminus S_1)$. 
If $ S\setminus S_1$ is
dependent, then we repeat the above argument to find a interval 
$S_2$. We continue until $ S \setminus ( \cup S_i) $ is independent
and then name this set $S_0$.   It is clear that we have the rank of
$S$ is the sum of the ranks  of the subsets. 
\end{proof}

\begin{proposition} 
\label{FullSet}
Suppose that $ \mspan (S) = S$,  $ e_1-e_2+\dots
  + e_ { 2m+1} \in S$ and 
$$
e_1 -e_2 + \dots + e_ { 2m+1} \in \mspan ( S\setminus \{ e_1-e_2+ \dots
+ e_ { 2m+1}\}).
$$
If $ S \setminus \{ e _1 -e_2 +\dots e_{2 m+1}\}  = \cup _i S_i$ and each
$S_i$ is an interval of the form $[e_s,e_t]$, then $S=B$. 
\end{proposition} 

\begin{proof} Since $ \mspan (S) =S$, if $ e_k -e _ {k+1} \not \in S$,
  then also $ e_k \not \in S$ or $ e_ { k+1} \not \in S$. If $ e_j$ is
  not in $ S$, then we have that $e_{j-1} -e_j$ and $e_j - e_ { j+1}$
  are not in $S_i$ for any $i$. But this implies that no vector in $S
  \setminus \{ e_ 1- e_2+ \dots+ e _ { 2m+1} \}$ has a non-zero $e_j$
  component and thus $e_1 -e_2 + \dots + e _ { 2m+1} $ is not in $
  \mspan (S\setminus \{ e_1-e_2+ \dots + e_ { 2m+1}\})$. 
\end{proof}

We are ready to give the proof of  our   Theorem. 

\begin{proof}[Proof of Theorem \ref{Inside}]
To show $ P_\delta \subset \Omega_M$, we use  the
characterization of the matroid polytope by the inequalities 
in (\ref{Rank}). 
Note that it suffices to  consider these inequalities for sets which
satisfy $ \mspan (S) = S$. 

We begin by considering sets $S \subset M \setminus \{ e_1 -e_2 +
\dots + e_ { 2m+1} \}$. By Lemma \ref{Structure}, we may write  $ S  =
 \cup _ { i=0 }^ k  S_i $ where the set $S_0$ is independent and each
 $ S_i$ is an interval  of the form $ [e_s, e_t]$.  
We let $L$ denote the cardinality of $S_0$. 
Using Lemma \ref{Structure} and then Lemma \ref{SegEst} for each of the
intervals in this decomposition,  we obtain 
\begin{eqnarray*}
\rank (S) & = & \sum _ { i = 0 } ^ k \rank (S_i) \\
  &  \geq &  \rank(S_0) + 
 k ( 1/2 -3 \delta ) + \sum  _ { i=1 } ^ k \sum _ { f_j \in S_ i }
 \theta _j\\
& \geq  & L( 1/2 -\delta ) + k ( 1/2-3\delta ) + \sum _ { f_j \in S}
\theta _j. 
\end{eqnarray*}
In the last inequality, we use that $S_0$ is independent and each $
\theta _j \leq 1/2 +\delta$. 
From this,  it is clear that  we have the inequality  (\ref{Rank})
when 
$ \delta \leq 1/6$. 

Now we consider the case when $  e_1-e_2+ \dots + e_ { 2m+1} \in S$
and thus we write 
$$
S = S' \cup \{ e_1-e_2+\dots+ e_ {2m+1}\}.
$$
If $ \{ e_1-e_2+ \dots e_ {2m+1}\} \not \in \mspan ( S')$, then we have 
$$
\rank(S') \geq \sum  _ {f_i \in S'} \theta _i 
$$
by the previous case the estimate (\ref{Rank}) for $S$ follows
since $\rank (S) = 1 + \rank(S') \geq \theta _{4m+2} + \rank (S')$.

Finally, we consider the case when $  e_1-e_2+ \dots + e_ { 2m+1} \in
\mspan(S')$. In this case, we write $ S'= \cup _ { i=0}^ k S_i$ as in
Lemma \ref{Structure}. 
 Using Lemma \ref{SegEst} we obtain 
$$
\rank(S) = \rank(S') \geq \sum _ { f_i \in S' } \theta _i + k ( 1/2 -3
\delta) + L ( 1/2-\delta). 
$$
If $L =0$, then $S = M$ by Proposition \ref{FullSet}  and thus we have
$\rank(S) = \sum _ {f_i \in S} \theta _i = 2m+1$ from the definition
of $P_ \delta$. 

In the remaining cases, we want $k ( 1/2-3 \delta) + L ( 1/2-\delta )
\geq \theta _ { 4m+2}$ which is implied by  
\begin{equation} \label{LastStep}
k ( 1/2- 3\delta ) + L ( 1/2- \delta ) \geq 1/2 + \delta. 
\end{equation}
If $ L =1$, then $ k \geq 1$ as otherwise $S'$ contains only one vector
and we cannot have $ e_1-e_2 + \dots + e_ { 2m+1}  \in \mspan(S
')$. If $ L =k=1$, then we have (\ref{LastStep}) if $\delta \leq
1/10$. If $k \geq 2$, then we have (\ref{LastStep}) if $ \delta \leq
1/6$. 
\end{proof} 

In order to apply Corollary \ref{LebesgueBound} to the form
associated to the matroid $M$, we will need to
compute the determinants arising in Proposition \ref{BaseBound} for
the matroid $M$.  
\begin{lemma} \label{Determinants}
 Let $ B\subset M$ be a basis and let $ \hat B$ be the
  matrix whose rows are the vectors in $B$. We have $|\det \hat B |
  =1$. 
\end{lemma}
\begin{proof} 
We begin by ordering the vectors in $B$ in the following way. We let
$f_{j_1}$ be the first vector on the list $ e_1, e_1-e_2, e_1-e_2+
\dots + e_{ 2m+1}$ that appears in $B$. Now given $ f_{j_1}, \dots,
f_{j _k}$, we choose $ f_{j _ { k+1}}$ to be the first vector on the
list $e_{ k}-e_{k+1}, e_{k+1} , e_{ k+1} -e_ {k+2}, e_1-e_2+ \dots + e_ { 2m+1}$
that is an element in the set $B \setminus \{ f_{j_1}, \dots, f_
{j_k}\}$.  We claim that this procedure continues until all of the
vectors in $B$ have been chosen. 

To establish the claim, we argue by contradiction. Suppose that for
some $k$, there
is no choice for $ f_{ j _{k+1}}$. We claim that   $ B
\setminus \{ f_{ j_1} , \dots, f_ {j _k }\}$ is contained in  the
span of $\{e_ { k+2}, \dots, e_ {2m+1} \}$.  If  we have this
containment, then the rank of $ B \setminus \{ f_{ j_1}, \dots, f_ { j
  _ k }\}$ is at most $2m-k$ and  the rank of $\{f_{j_1} ,\dots, f_{
  j_k}\}$ is at most $k$ and we obtain a contradiction with our
assumption that $B$ is a basis. 
Because we are assuming there is no choice for $ f_{j_{k+1}}$, the
vectors
$\{e_k-e_{ k+1}, e_k , e_{ k+1} - e_ { k+2},  e_1-e_2+ \dots+e_{
  2m+1}\}$ are not in $
B\setminus \{ f_{j_1}, \dots, f_{j_k}\}.$ In addition, none of the
vectors $e_{ i-1} -e_{ i}$, $i=2,\dots, k$  can be in 
$B\setminus \{ f_{j_1}, \dots, f_{j_k}\}$ as the vector $e_{ i-1}
-e_i$ has first priority when we choose $f_{j_i}$. For the same
reason, we do not have $ e_1$ in 
$B\setminus \{ f_{j_1}, \dots, f_{j_k}\}$. 
Finally, suppose
for some $i$, $2\leq i \leq k $,  $
e_i$
is in 
$B\setminus \{ f_{j_1}, \dots, f_{j_k}\}$. This implies $ f_{j_i}= e _
{ i-1} - e_i$  as this is the only vector with higher priority than
$e_i$. Working backwards, we see that $f_{ j_{i-1}} $ is either $ e_{
  i-1}$ or $ e_{ i-2}  -e _{ i-1}$ and continuing we find that for
some $j$ with $ 1\leq j < i$, we  have the vectors 
$ e_ j, e_{j}- e_{ j+1}, e_{ j+1}-e_{ j+2}, \dots e_{ i-1} -e _i ,e _i$  in
$B$. This is a dependent set of vectors and contradicts our assumption
that $B$ is basis. Thus our claim holds.

We let $ \hat B$ be the matrix whose rows are the vectors $ f_ {j _1},
\dots f_ {j _ {2m+1}}$. We claim that $| \det \hat B| = 1$ and
consider several cases to give the proof.

{\em Case 1. } Suppose $e_1-e_2+\dots  +  e_ { 2m+1}$ is not in $B$.

In this case, we show how to use column operations to reduce $ \hat B$
to a lower triangular matrix. Suppose $ \hat B _ { i, i+1} =0 $ for $
i = 1, \dots , k-1$ and that $ \hat B_ { k, k+1} \neq 0$. In this
case, we have $ f_ {j _ k} = e _ k -e_ { k+1}$ and $ \hat B _ { k+1,
  k} = 0 $ since we have $ f_ {j _ { k+1}} \neq e_ k - e_ { k+1} $ and
$f_{ j _ { k+1} } \neq e _ 1- e_ 2 + \dots + e_ { 2m+1}$.  We replace
the $(k+1)$st column, $\hat B_ {\cdot, k+1}$ by the sum $ \hat B _ {
  \cdot , k } + \hat B_ { \cdot , k+1}$ and obtain a matrix with $
\hat B _ {k, j}=0 $ for $j \leq k -1$ and $ \hat B _ { k,k} = \pm
1$. Continuing in this manner, we obtain a lower triangular matrix
with entries of $ +1$ or $-1$ on the diagonal. It follows that $ \det
\hat B = \pm 1$.

{\em Case 2. }  Suppose $e_1-e_2+\dots e_ { 2m+1}$ is  in $B$.

In this case, we fix $k$ so that $ f_{j _k} = e _1 -e_2 + \dots e_
{2m+1} $ and write the matrix 
$$
\hat B = \left ( \begin{array}{cc}A & C \\ 0&  D \end{array} \right )
$$
where the block $A$ is of size $ k\times k$, $C$ is of size $ k \times
(2m+1 -k)$ and $D$ is of size $ (2m+1-k) \times ( 2m + 1-k)$.  Note
that our ordering of the basis guarantees that the lower left block
is 0.  We may
apply the same argument used above and find column operations which
reduce the matrix 
$A$ to a lower triangular matrix with
diagonal entries of $ \pm 1$.  Observe that as we are either leaving
column $i$ unchanged 
 or replacing  column $i$  by the sum of  column $i$ and $i-1$,
the entries in the $k$th row  $ \hat B_ {k,i}$, $i=1,\dots,k$ will be
either $0$, $1$ or $-1$. Since we assume that $B$ is a basis, we
cannot have $ \hat B _{k,k}=0$. We apply the same procedure to reduce
the block $D$ to a   lower triangular matrix with diagonal entries of
$ +1$ or $ -1$. Since the blocks $A$ and $D$ have determinant $ \pm
1$, it follows that the determinant of the matrix $ \hat B =0$. 
\end{proof}

\note {
Check the lemma about determinants. $ \det \left( \begin{array}{cc} A & C
\\ 0 & D \end{array} \right) = \det A \det D. $
}

As a consequence of  the previous lemma, we immediately obtain the
following estimate for the  form (\ref{GeneralForm}) specialized to
the matroid we are studying in this section 
$$
\Lambda (a_1, \dots a_{ 4m+2} ) \leq \prod _ { i=1}^ { 4m+2} \| a_i 
\|_{ \theta _i }, \qquad \theta \in \Omega _M. 
$$

Finally, we are ready to give the proof of our main theorem.

\begin{theorem} Suppose that $ \theta $ lies in the interior of $ P_
  \delta$ and that the indices $q_1, \dots q_ {4m+2}$ satisfy $  \sum
  _ {j=1}^ 4 1/q _ { 4k+j} \geq 1$ for $ k =0,\dots ,m-1$. Then we may find a constant $c=c_\theta$ so that 
\begin{eqnarray*}
\lefteqn{\Lambda (a_1, \dots, a_ { 4m+2}) \leq c^ n \| a_ {4m+1} \|_ { 
    \theta _ { 4m+1}} \|a_{4m+2}\|_{\theta _ {4m+2}}} \\
 & & \cdot \prod _ { j = 0} ^ { m-1} (
\| a_ { 4j+1}\| _ { \theta _{ 4j+1}, q_{ 4j+1}}
\| a_ { 4j+2}\| _ { \theta _{ 4j+2}, q_{ 4j+2}}
\| a_ { 4j+3}\| _ { \theta _{ 4j+3}, q_{ 4j+3}}
\| a_ { 4j+4}\| _ { \theta _{ 4j+4}, q_{ 4j+4}}
)
\end{eqnarray*}
The constant $c$ depends on $ \max \{ 1/( 1/10-|\theta _i -1/2|): i=1,
\dots, 4m+2\}$.
\end{theorem}

\begin{proof} By Theorem \ref{Inside}, we have $ P _ \delta \subset
  \Omega _M$  if $ \delta \leq 1/10$.  Thus, we have that $ \theta $
  is an interior point of $ \Omega_M$. We will prove by induction that 
if $ \eta \in P_\delta $ and $ \eta _i = \theta _i$ for $ i = 1, \dots
, 4k$, then we have 
\begin{eqnarray}\label{InductionStep}
\lefteqn{ \Lambda (a_1, \dots, a_ { 4m+2})  \leq  c^ k 
\prod _ { i=4k+1} ^ {4m+2} \|a_i \|_{ \eta_i}  }
 & &  \\
&& \nonumber
\cdot \prod _ { j = 0} ^ { k-1} (
\| a_ { 4j+1}\| _ { \theta _{ 4j+1}, q_{ 4j+1}}
\| a_ { 4j+2}\| _ { \theta _{ 4j+2}, q_{ 4j+2}}
\| a_ { 4j+3}\| _ { \theta _{ 4j+3}, q_{ 4j+3}}
\| a_ { 4j+4}\| _ { \theta _{ 4j+4}, q_{ 4j+4}}
)
\end{eqnarray}
We use $k=0$ as the base case. The estimate  we need holds for $
\theta \in \Omega_M$ and follows from 
 Corollary \ref{LebesgueBound}  and Lemma \ref{Determinants}.

Now suppose that the estimate (\ref{InductionStep}) holds for $k< m$ and
we show how to obtain the same result for $k+1$. Fix $ a_1, \dots , a_
{4k} , a_{4k+5}, \dots, a_{4m+2}$ and set 
$$
\Lambda_0( a_{4k+1}, \dots, a_{4k+4})
= \Lambda (a_1, \dots, a_{4m+2}).
$$
We consider the three directions 
\begin{eqnarray*}
u^1 & =& ( 1,1,-1,-1)\\
u^2 & =& ( 1,-1,1,-1)\\
u^3  & =& ( 1,-1,-1,1). 
\end{eqnarray*}
We will need the six points 
$$
(\theta _{ 4k+1}, \theta _ {4k+2}, \theta _ { 4k+3}, \theta _ {4k+4})
\pm \tau u^j, \qquad j =1,2,3
$$
where $ \tau = \min \{ 1/10 - | \theta _ { 4k+i} -1/2| : i =1,\dots,
4\}$.
Each of these six points lies in $ P_ \delta$. As the vectors $ u_j$
give three linearly independent directions, the convex hull of these
six points give us a neighborhood of $  (\theta _{ 4k+1}, \theta _
{4k+2}, \theta _ { 4k+3}, \theta _ {4k+4})$ in $ P_\delta$. Applying 
our induction hypothesis and then  Theorem \ref{Janson} gives
that 
\begin{eqnarray*}
\lefteqn{
\Lambda_0( a_{4k+1}, \dots, a_{4k+4})
\leq c^{k+1} 
\prod _ { i=4k+4} ^ {4m+2} \|a_i \|_{ \eta_i}  
} 
&& \\
&&
\cdot  \prod _ { j = 0} ^ { k} (
\| a_ { 4j+1}\| _ { \theta _{ 4j+1}, q_{ 4j+1}}
\| a_ { 4j+2}\| _ { \theta _{ 4j+2}, q_{ 4j+2}}
\| a_ { 4j+3}\| _ { \theta _{ 4j+3}, q_{ 4j+3}}
\| a_ { 4j+4}\| _ { \theta _{ 4j+4}, q_{ 4j+4}}
)
\end{eqnarray*}
The Theorem now follows by induction. 
\end{proof}
Finally, we observe that this theorem implies the following estimate
for the  form $ \Lambda _n$ defined   in  (\ref{MainFormDef}). 
\begin{corollary}
If $ 1/p + 1/p' =1$, and  $|1/p-1/2|< 1/10$ we have
$$
\Lambda_n (t, q_0, q_1, \dots, q_{2n})\leq c^n \|t\|_{1/p} \|q_0\|_{1/p'}\prod _ { j =0
} ^ {2n} \| q_j \|_{1/2}.
$$
\end{corollary} \label{FormCorollary}
\begin{proof} We observe that in our previous Theorem, we may let $
  \theta _j = 1/2$ for $j = 1,\dots 4m$. The functions $ a_ { 2j} $,
  $j = 1, \dots, m$ are chosen to be $ 1/x$ or $ 1/\bar x$ which lie
  in $L^ { 2, \infty }( \complexes)$. With these choices, the estimate
  follows immediately.
\end{proof}

Let $ {\cal T }$ be the map that takes a potential
  $Q$ to the scattering data $S$ as defined, for example, in Beals and
  Coifman \cite{BC:1988} or Sung \cite{LS:1994a,LS:1994b,LS:1994c}. 
Combining the estimate of Corollary
  \ref{FormCorollary} with the method of proof in the work of Brown
  \cite{RB:2001b}, we obtain the following result.
\begin{corollary} Let $ 1/10> 1/p -1/2  \geq 0$, then there exists
  $N$, a
  neighborhood of 0 in $ L^ p ( \complexes ) \cap L^2 ( \complexes )$ so
  that 
$$
\| 
{\cal T}(q) \|_{ 1/p'} \leq  \frac C { 1 - c^2\|q\|^2_{ 1/2}}\|q\|_ { 1/p}.
$$
\end{corollary} 
\bibliographystyle{plain}
\def\cprime{$'$} \def\cprime{$'$} \def\cprime{$'$}

\small \noindent \today 
\end{document}